\documentclass[11pt,a4paper,leqno,twoside]{amsart}

\usepackage{amsmath}
\usepackage{amssymb}
\usepackage{amsthm}
\newcommand{\de}{\delta}
\newcommand{\e}{\varepsilon}
\newcommand{\la}{\lambda}
\newcommand{\R}{\mathbb R}
\newtheorem{theorem}{Theorem}[section]
\newtheorem{lemma}[theorem]{Lemma}
\newtheorem{remark}[theorem]{Remark}
\newtheorem{proposition}[theorem]{Proposition}

\newtheorem{corollary}[theorem]{Corollary}
\numberwithin{equation}{section}
\usepackage{color}

\begin{document}
\title[Blow-up analysis]{Blow-up analysis for nodal radial solutions in Moser-Trudinger critical equations in $\R^2$}
\author[M. Grossi and D. Naimen]{Massimo Grossi$^\dag$ and  Daisuke Naimen$^\ddag$}
\thanks{The first author is  supported by PRIN-2015 grant (P.I. A. Malchiodi).
}
\address{$\dag$ Dipartimento di Matematica, Universit\`a di Roma ``La Sapienza'',      \texttt{massimo.grossi@uniroma1.it}}
 \address{$\ddag$ Muroran Institute of Technology, \texttt{naimen@mmm.muroran-it.ac.jp}}
\thanks{The second author was supported by Grant-in-Aid for JSPS Research Fellow Grant Number JP15J12092.
}
\begin{abstract}
In this paper we consider nodal radial solutions $u_\e$ to the problem
\[
\begin{cases}
-\Delta u=\lambda ue^{u^2+|u|^{1+\e}}&\text{ in }B,\\
u=0&\text{ on }\partial B.
\end{cases}
\]
and we study their asymptotic behaviour as $\e\searrow0$.\\ 
 We show that when $u_\e$ has $k$ interior zeros, it exhibits a multiple blow--up behaviour in the first $k$ nodal sets while it converges to the least energy solution of the problem with $\e=0$ in the $(k+1)$--th one. We also prove that in each concentration set, with an appropriate scaling, $u_\e$ converges to the solution of the classical Liouville problem in $\R^2$.
\vskip0.2cm
{\bf Keywords:} Radial solutions, Moser-Trudinger inequality.

{\bf AMS Subject Classifications:} 35B32, 35J61.
\end{abstract}

\maketitle

\section{Introduction}
In this paper we study the asymptotic behavior of  nodal radial solutions to
\begin{equation}
\begin{cases}\label{P}
-\Delta u=f_\e(u)&\text{ in }B\\
u=0&\text{ on }\partial B,
\end{cases}
\end{equation}
where $f_\e(s)$ is some smooth nonlinearity of Moser-Trudinger type depending on 
a  positive parameter $\e$ goings to zero. In what follows we will be more precise on the conditions on $f_\e(s)$. In \eqref{P} $B$ is the unit ball of $\R^2$. \\
First let us recall some classical results when $u$ is a {\em positive} solution of \eqref{P}. In this case our problem is linked to the celebrated Moser-Pohozaev-Trudinger inequality (\cite{M}, \cite{P}, \cite{T},) namely
\begin{equation}\label{i1}
\sup\limits_{\int_B|\nabla u|^2\le1}\int_Be^{4\pi u^2}
\le c
\end{equation}
for any $u$ in the Sobolev space $W^{1,2}_0(B)$.\\
In the pioneering paper \cite{CC} it was proved that the supremum in \eqref{i1} is achieved and  the corresponding Eulero-Lagrange equation is given by
\begin{equation}\label{i2}
\begin{cases}
-\Delta u=\frac{ue^{u^2}}{\int_Bu^2e^{u^2}}&\text{ in }B\\
u=0&\text{ on }\partial B,
\end{cases}
\end{equation}
Now we consider the related (but not equivalent) problem
\begin{equation}\label{i3}
\begin{cases}
-\Delta u=\la ue^{u^2} &\text{ in }B\\
u=0&\text{ on }\partial B.
\end{cases}
\end{equation}
In \cite{OS} there is an interesting discussion on relationship between \eqref{i2} and \eqref{i3}.\\
In \cite{A} it was proved the existence of solutions to  \eqref{i3} for any $\la\in(0,\la_1)$ where $\la_1$ is the first eigenvalue of $-\Delta$ with Dirichlet boundary conditions (see also \cite{FMR}). As $\la\rightarrow0$, the corresponding solution $u_\la$ concentrates around the origin and its asymptotic behavior was studied in \cite{OS} and  \cite{AD}.\\
These results hold also for more general problems like
\begin{equation}\label{i6}
\begin{cases}
-\Delta u=\la f(u)e^{u^2} &\text{ in }B\\
u=0&\text{ on }\partial B.
\end{cases}
\end{equation}
We refer to \cite{A} and  \cite{FMR} for the precise assumptions on $f$.\\
If we consider the case of sign changing {\em radial} solutions  we find some interesting differences. Indeed, in \cite{AY} the authors showed that in order to have existence results of nodal solutions we need to impose  some restrictions on the nonlinearity $f$ in \eqref{i6}. A particular case is the following,
\begin{theorem}\label{tt1}(See \cite{AY} and \cite{AY1})
Let us consider the problem
\begin{equation}\label{i7}
\begin{cases}
-\Delta u=\la ue^{u^2+|u|^\beta} &\text{ in }B\\
u=0&\text{ on }\partial B.
\end{cases}
\end{equation}
Then we have that,\vskip0.2cm
$i)$ if $1<\beta<2$ there exists a radial solution with $k$ interior zeros for any integer $k\ge1$ and for any $\la\in(0,\la_1)$,\vskip0.2cm
$ii)$ if $0\le\beta\le1$ there exists $\la=\la_{AY}>0$ such that for any $0<\la<\la_{AY}$ there exist no solution to \eqref{i7}.
\end{theorem}
From this result we get that the nonlinearity $f(s)=se^{s^2+s}$ is the {\em border line case} between the existence and nonexistence of nodal solutions.  Hence it becomes interesting to study the asymptotic behavior of the solution $u_\e$ in  \eqref{i7} as $\beta=1+\e$, $0<\la<\la_{AY}$ and $\e\searrow0$. \\
In order to state our main result we need to introduce some notations. First let us denote by $u_0$ the solution of
\begin{equation}
\begin{cases}\label{i9}
-\Delta u=\lambda ue^{u^2+u}&\text{ in }B,\\
u>0&\text{ in }B,\\
u=0&\text{ on }\partial B.
\end{cases}
\end{equation}
Next, for $u\in H^1_0(B)$ and $\e\ge0$ let us consider the functional
\begin{equation}
I_\e(u)=\frac12\int_B|\nabla u|^2-\int_BF_\e(u)
\end{equation}
where $F_\e(s)=\lambda\int_0^ste^{t^2+|t|^{1+\e}}dt$.  We have the following result,
\begin{theorem}[Global behavior]\label{i10}
Let $u_\e$ be a nodal radial solution obtained in \cite{AY1} which verifies
\begin{equation}
\begin{cases}\label{i11}
-\Delta u=\lambda ue^{u^2+|u|^{1+\e}}&\text{ in }B,\\
u=0&\text{ on }\partial B.
\end{cases}
\end{equation}
with $k$ interior zeros denoted by $0=r_{0,\e}<r_{1,\e}<r_{2,\e}<\dots<r_{k,\e}<1=r_{k+1,\e}$. Assume that $u_\e(0)>0$.
Then we have that, as $\e\rightarrow0$ and $0<\la<\la_{AY}$,
\begin{equation}\label{i99}
u_\e(x)\rightarrow(-1)^ku_0(x)\quad\hbox{ in }C_{\text{loc}}^2(B\setminus\{0\})
\end{equation}
\begin{equation}\label{i98}
r_{i,\e}\rightarrow0\quad\hbox{ for any }i=1,\dots,k,
\end{equation}
\begin{equation}\label{i97}
\int_B|\nabla u_\e|^2\rightarrow \int_B|\nabla u_0|^2+4k\pi,
\end{equation}
\begin{equation}\label{i96}
I_\e(u_\e)\rightarrow I_0(u_0)+2k\pi.
\end{equation}
\end{theorem}
\begin{theorem}[Local behavior]\label{i12}
Let $u_\e$ be the solution considered in the previous theorem. Then for $i=0,\dots,k$ let $M_{i,\e}\in(r_{i-1,\e},r_{i,\e})=A_{i,\e}$  be the points such that $u_\e(M_{i,\e})=||u_\e||_{L^\infty(A_{i,\e})}$ (we have that $M_{1,\e}=0$). Then if $\de_{i,\e}$  is defined as $\delta_{i,\e}=r_{i,\e}\gamma_{i,\e}$ with
\begin{equation}
2\la r_{i,\e}^2e^{||u_\e||_{L^\infty(A_{i,\e})}^{2}+||u_\e||_{L^\infty(A_{i,\e})}^{1+\e}}||u_\e||_{L^\infty(A_{i,\e})}\gamma_{i,\e}^2=1
\end{equation}
we have that $\delta_{i,\e}\to 0$ and
\begin{equation}
\begin{split}
2||u_\e||_{L^\infty(r_{i-1,\e},r_{i,\e})}(u_\e\left(M_{i,\e}+\de_{i,\e}r\right)&-||u_\e||_{L^\infty(r_{i-1,\e},r_{i,\e})})\\&\rightarrow\log\frac1{\left(1+\frac{r^2}8\right)^2}\quad\hbox{in }C^1_{loc}(0,+\infty).
\end{split}
\end{equation}
\end{theorem}
\begin{remark}
Another interesting problem with similar behavior is given by
\begin{equation}
\begin{cases}
-\Delta u=\lambda ue^{u^{2-\e}}&\text{ in }B,\\
u=0&\text{ on }\partial B.
\end{cases}
\end{equation}
As for \eqref{i11} it is possible to show that there exists a family of nodal solutions $u_\e$ for any $\e>0$. Despite the nonlinearity is not covered by the assumptions of Theorem 1.3 in \cite{AY1} we can still repeat the proof in order to get the existence result. Moreover the result in \cite{AY} applies and so there exists a constant $\bar\la$ such that for any $\la\in(0,\bar\la)$ there exists no sign changing solution.\\
It is possible to show that analogous results like in Theorems \ref{i10} and  \ref{i12} hold. The interest in this type of nonlinearity is given by the similarity with the analogous in higher dimension (see problem \eqref{i16} and the comments below).
\end{remark}
\begin{remark}
Similar phenomena to Theorem \ref{tt1}, \ref{i10} and \ref{i12} appears in higher dimensions for the problem
\begin{equation}
\begin{cases}
-\Delta u=|u|^\frac4{N-2}u+\la u&\text{ in }B,\\
u=0&\text{ on }\partial B.
\end{cases}
\end{equation}
where $N\ge3$ and $B$ is the unit ball of $\R^N$.\\
In \cite{ABP} it was proved that if $N=4,5,6$ there exists $\la^*>0$ such that there is no nodal radial solution for $0<\la<\la^*$. The asymptotic behavior of the solution $u_\la$ as $\la\to\overline{\la}$ for a limit value $\overline{\la}>0$ and $N=4,5,6$ was studied in \cite{IP}. Note that the case $N=6$ has strong similarities with our results for the case $k=1$. Other existence results for $N=4,5$ can be founded in \cite{IV}.
\end{remark}
It is interesting to compare the previous results with other similar problems like
\begin{equation}\label{i13}
\begin{cases}
-\Delta u=|u|^{p-1}u\text{ in }B\subset\R^2\\
u=0\text{ on }\partial B,
\end{cases}
\end{equation}
(see \cite{GGP}) and
\begin{equation}\label{i14}
\begin{cases}
-\Delta u=\lambda\sinh\text{ in }B\subset\R^2\\
u=0\text{ on }\partial B.
\end{cases}
\end{equation}
(see \cite{GP}).\\
Both this problems share the feature that suitable transformations of positive solutions converge to the limit problem 
\begin{equation}\label{i15}
\begin{cases}
-\Delta u=e^u\text{ in }\R^2\\
\int_{\R^2}e^u<+\infty
\end{cases}
\end{equation}
We want to compare Theorems \ref{i10} and  \ref{i12} with the analogous ones for  \ref{i13} and  \ref{i14}.\\
The global behavior is different: indeed solutions founded in \cite{GP} converge to  suitable multiple of the Green function which does not belong to $W^{1.2}_0(B)$ and solutions studied in \cite{GGP} goes to $0$ everywhere.\\
However more striking differences appear if we look at the local behavior. Indeed, in this case the solutions to  \eqref{i13} and  \eqref{i14} involve the singular Liouville problem
\begin{equation}
\begin{cases}
-\Delta u=|x|^\alpha e^u\text{ in }\R^2\\
\int_{\R^2}|x|^\alpha e^u<+\infty,
\end{cases}
\end{equation}
for some suitable positive number $\alpha$. We refer to \cite{GGP} and \cite{GP} for more precises  statements. In our case the local behavior of the solution is again related to the problem \eqref{i15}. In some sense our problem is more similar to the ``almost critical'' problem in higher dimensions $N\ge3$ given by
\begin{equation}\label{i16}
\begin{cases}
-\Delta u=|u|^{\frac4{N-2}-\e}u\text{ in }B\subset\R^N\\
u=0\text{ on }\partial B.
\end{cases}
\end{equation}
In this case the local behavior of nodal solutions is given by the (unique) positive smooth solution of the limit problem (see \cite{BEP}, \cite{DIP} and \cite{PW})
\begin{equation}
-\Delta u=u^\frac{N+2}{N-2}\text{ in }\R^N.
\end{equation}
In our opinion this similarity is due to the effect of nonlinearity which is very close to those in Moser-Trudinger inequality.\\
The paper is organized as follows: in Section \ref{s0} we prove some energy estimates for the solution $u_\e$. In Section \ref{sec:ball} we study the behavior of $u_\e$ in the ball $B_{r_{1,\e}}$ where $r_{1,\e}$ is the first zero of $u_\e$. In Section \ref{sec:annulus} and \ref{sec:outer} we consider the behavior of $u_\e$ in the other annular regions and in Section \ref{sec:proof} we give the proof of Theorems \ref{i10} and  \ref{i12}. Finally in Appendix \ref{sec:basic} we prove some technical lemmas.
\vskip0.1cm
For all $u\in H^1_0(B)$, we define $\|u\|:=\left(\int_{B}|\nabla u|^2dx\right)^{1/2}$. In addition, let $B(0,r):=B_r$ and $B(r,s):=B_{s}\setminus B_{r}$ for $r,s>0$. 
\section{Energy estimates for $u_\e$}\label{s0}
In the following, we always assume $0<\lambda<\min\{\lambda_1,\lambda_{\text{AY}}\}$ and we consider the least energy nodal solution $u_\e$ of \eqref{i11} obtained by Theorem 1.3 in  \cite{AY1}. More precisely, we define $H_{\text{r},0}^1(B)$ as a subspace of $H^1_0(B)$ which consists of all the radial functions and by the Nehari manifold
\[
\mathcal{N}_{\e}=\left\{u\in H^1_{\text{r},0}(B)\setminus\{0\}\ |\ \int_{B}|\nabla u|^2dx=\int_Bf_{\e}(u)udx\right\},
\]
and for $k\in\mathbb{N}$, 
\[
\begin{split}
\mathcal{N}_{k,\e}:=&\Big\{u\in H^1_{\text{r},0}(B)\ |\ \exists r_i\in(0,1);\ 0=r_0<r_1<\cdots<r_{k+1}=1, \\&\ \ u(r_i)=0,\  u_i:=u|_{B_{(r_{i-1},r_i)}},(-1)^{i-1}u_i>0,\ u_i\in \mathcal{N}_\e,\ 1\le i\le k+1\Big\}.
\end{split}
\]
Then let $u_{\e}\in \mathcal{N}_{k,\e}$ be a solution to \eqref{i11} such that 
\[
I_{\e}(u_{\e})=\inf_{u\in \mathcal{N}_{k,\e}}I_{\e}(u).
\]
We choose constants $0=r_{0,\e}<r_{1,\e}<\cdots<r_{k,\e}<r_{k+1,\e}=1$ so that $u_\e(r_{i,\e})=0$ for $i=1,2,\cdots,k$. Moreover, for each $i=1,2,\cdots,k+1$, define $u_{i,\e}:=u_\e|_{B(r_{i-1,\e},r_{i})}$ with zero extension to whole $B$.

First let us show a suitable upper bound for $I_\e(u_{\e})$. To this end, we  use the Moser function defined in \cite{A1}. For $0<l<R\le1$, we define 
\[
m_{l,R}(x):=\frac{1}{\sqrt{2\pi}}\begin{cases} \left(\log{\frac{R}{l}}\right)^{\frac{1}{2}}&\ 0\le|x|<l\\
\frac{\log{\frac{R}{|x|}}}{\left(\log{\frac{R}{l}}\right)^{\frac{1}{2}}}&\ l\le |x|\le R\\
0&\ |x|>R.
                                    \end{cases}
\]
Then it satisfies $m_{l,R}\in H^1_0(B)$ and $\|m_{l,R}\|=1$. In addition let us define a cut off function,
\[
\phi_{l,R}(x)=1-\frac{m_{l,R}(x)}{\sqrt{2\pi}^{-1}\left(\log{\frac{R}{l}}\right)^{\frac{1}{2}}}\in H^1(B)
\]
Then we have $0\le \phi_{l,R}\le 1$, $\phi_{l,R}=0$ on $B_{l}$ and $\phi_{l,R}=1$  on $B\setminus B_{R}$. 
For $0<l_1=l_{1,\e}<R_1=R_{1,\e}<p_1=p_{1,\e}<l_2=l_{2,\e}<R_2=R_{2,\e}<p_2=p_{2,\e}<\cdots<l_k=l_{k,\e}<R_k=R_{k,\e}<1$, we set 
\[
\begin{cases}
w_{1,\e}:=m_{l_1,R_1},\\w_{i,\e}:=(-1)^{i-1}\phi_{R_{i-1},p_{i-1}} m_{l_i,R_i}\text{ for }k=2,\cdots,k,\text{ and }\\w_{k+1,\e}:=(-1)^{k}\phi_{R_k,1} u_0,
\end{cases}
\]
where $u_0$ is the least energy solution of \eqref{i9} obtained in \cite{A1} and thus, it satisfies  
\[
I_0(u_0)=\inf_{u\in \mathcal{N}_0}I_0(u)\in(0,2\pi).
\]
 We choose $l_1,R_1,p_1,\cdots,l_k,R_k$ so that $R_{k}\to0$ and
\begin{equation}
\begin{cases}
\frac{\log{\frac{1}{R_i}}}{\log{\frac{1}{l_i}}}\to 0\ (i=1,2,\cdots,k),\\
\frac{\log{\frac{R_{i}}{l_{i}}}}{\log{\frac{p_{i-1}}{R_{i-1}}}}\to 0,\ \ \frac{p_{i-1}}{l_i}\to0\ (i=2,\cdots,k),
\end{cases}\label{eq:parameters}
\end{equation}
as $\e\to0$. For example, take any $R_k>0$ such that $R_k\to0$ as $\e\to0$ and then,  choose $l_k=e^{-1/R_k}$, $p_{k-1}=l_k^2$, and $R_{k-1}=p_{k-1}e^{-1/l_k}$. Similarly, set $l_{k-1}=e^{-1/R_{k-1}}$, $p_{k-2}=l_{k-1}^2$, $R_{k-2}=p_{k-2}e^{-1/l_{k-1}}$ and so on. 
 We note that, for every $i=1,2,\cdots,k+1$ and $\e\in (0,1)$, there exists a constant $t_{i,\e}>0$ such that $t_{i,\e}w_{i,\e}\in \mathcal{N}_{\e}$. (See Step 2 in the proof of Lemma 3.4 in \cite{A1}.) We define a test function 
\[
w_{\e}(x):=\sum_{i=1}^{k+1}t_{i,{\e}}w_{i,\e}.
\]
Then we have $w_\e\in \mathcal{N}_{k,\e}$. We obtain the following.
\begin{lemma}\label{lem:energyup}
We get 
\[
\limsup_{\e\to 0}I_{\e}(u_{\e})\le 2\pi k+I_0(u_0).
\] 
\end{lemma}
\begin{proof} 
First observe that since $w_{\e}\in \mathcal{N}_{k,\e}$, we have
\[
I(u_{\e})\le I_{\e}(w_{\e})=\sum_{i=1}^{k+1}I_{\e}(t_{i,\e}w_{i,\e}).
\]
Then it suffices to show, 
\begin{enumerate}
\item[(I)] $\limsup_{\e\to 0}I_{\e}(t_{1,\e}w_{1,\e})\le2\pi$,
\item[(II)] $\limsup_{\e\to 0}I_{\e}(t_{i,\e}w_{i,\e})\le2\pi$, \text{ for }$i=2,\cdots,k$, and,
\item[(III)] $\limsup_{\e\to 0}I_{\e}(t_{k+1,\e}w_{k+1,\e})\le I_0(u_0)$.
\end{enumerate}
(I) We claim 
\begin{equation}
\limsup_{\e\to 0}t_{1,\e}^2\le 4\pi.\label{eq;energy4}
\end{equation}
If not, there exist a sequence $(\e_n)$ and a constant $\delta>0$ such that $\e_n\to 0$ as $n\to \infty$ and $t_{1,\e_n}^2\ge 4\pi(1+\delta)$ for all $n$. Set $t_n:=t_{1,\e_n}$, $w_{n}:=w_{1,\e_n}$, $l_n:=l_{1,\e_n}$ and $R_n:=R_{1,\e_n}$ for simplicity. Since $t_nw_n\in \mathcal{N}_{\e_n}$, we get
\[
t_n^2\|w_n\|^2=\lambda \int_{B}(t_nw_n)^2e^{|t_nw_n|^{2}+|t_nw_n|^{1+\e_n}}dx.
\]
Then we have
\[
\begin{split}
t_n^2&\ge \lambda\int_{B_{l_n}}(t_nw_n)^2e^{|t_nw_n|^{2}+|t_nw_n|^{1+\e_n}}dx\\
         &\ge\frac{\lambda}{2}t_n^2 l_n^2\log{\frac{R_n}{l_n}}e^{\frac{t_n^{2}}{2\pi}\log{\frac{R_n}{l_n}}}\\
&=\frac{1}{2}t_n^2\log{\frac{R_n}{l_n}}\exp{\left\{\frac{t_n^{2}}{2\pi}\left(\log{\frac{1}{l_n}}-\log{\frac{1}{R_n}}\right)-2\log{\frac{1}{l_n}}\right\}}.
\end{split}
\]
Here, \eqref{eq:parameters} implies that
\[
\log{\frac{1}{R_n}}=o\left(\log{\frac{1}{l_n}}\right).
\]
It follows that
\[
\log{\frac{R_n}{l_n}}\to \infty \ (n\to \infty).
\] 
As a consequence, we find a constant $\delta'>0$ such that
\[
\begin{split}
2\ge \exp{\left\{\delta'\log{\frac{1}{l_n}}\right\}}
\end{split}
\]
for large $n$. Taking $n\to \infty$, we have a contradiction. Now, since $t_{1,\e}w_{1,\e}\in \mathcal{N}_\e$,  $\|w_{1,\e}\|=1$ and  $\limsup_{\e\to0}t_{1,\e}^2\le4\pi$, we get
\[
\left|\int_{B}f_{\e}(t_{1,\e}w_{1,\e})t_{1,\e}w_{1,\e}dx\right|\le C
\]
for some constant $C>0$ uniformly for $\e>0$. Furthermore, note $t_{1,\e}w_{1,\e}\to 0$ a.e. on $B$. Then by Lemma \ref{lem:F} in Appendix \ref{sec:basic}, we find  
\[
\lim_{\e\to0} \int_{B}F_{\e}(t_{1,\e}w_{1,\e})dx= \int_BF_0(0)dx=0. 
\]
As a consequence, we get
\[
\limsup_{\e\to0} I_{\e}(t_{1,\e}w_{1,\e})=\limsup_{\e\to0}\frac{t_{1,\e}^2}{2}\le2\pi.
\]
This finishes the proof of (I).\\ \ \\
(II) Fix $i=2,3,\cdots,k$. We first claim $\lim_{\e\to 0}\int_{B}|\nabla w_{i,\e}|^2dx= 1$. In fact, we get
\[
\begin{split}
\int_{B}|\nabla w_{i,\e}|^2dx&=\int_{B}|\nabla \phi_{R_{i-1},p_{i-1}}|^2m_{l_i,R_i}^2dx+2\int_{B}\phi_{R_{i-1},p_{i-1}}m_{l_i,R_i}\nabla \phi_{R_{i-1},p_{i-1}}\nabla m_{l_i,R_i}dx\\&+\int_{B}|\nabla m_{l_i,R_i}|^2\phi_{R_{i-1},p_{i-1}}^2dx\\
& =I_1+I_2+I_3.
\end{split}
\]
It follows from \eqref{eq:parameters} that 
\[
I_1=\int_{B(R_{i-1},p_{i-1})}|\nabla \phi_{R_{i-1},p_{i-1}}|^2m_{l_i,R_i}^2dx= \frac{\log{\frac{R_i}{l_i}}}{\log{\frac{p_{i-1}}{R_{i-1}}}}\to 0
\]
as $\e\to0$. Since $\phi_{R_{i-1},p_{i-1}}m_{l_i,R_i}\nabla \phi_{R_{i-1},p_{i-1}}\nabla m_{l_i,R_i}=0$ on $B$, we get $I_2=0$. Furthermore, as $\phi_{R_{i-1},p_{i-1}}=1$ on $B(l_i,R_i)$ and $\nabla m_{l_i,R_i}=0$ on $B_{l_i}$, we clearly have
\[
I_3=\int_{B}|\nabla m_{l_i,R_i}|^2dx=1.
\]
This shows the claim. Now we shall show $\limsup_{\e\to 0}t_{i,\e}^2\le 4\pi$. If not, there exists a constant $\delta>0$ such that $t_{i,\e}^2\ge 4\pi (1+\delta)$ for all small $\e>0$ by extracting a sequence if necessary. Then noting $t_{i,\e}w_{i,\e}\in \mathcal{N}_{\e}$ and \eqref{eq:parameters}, we get
\[
\begin{split}
1+o(1) &=\lambda \int_B (\phi_{R_{i-1},p_{i-1}}m_{l_i,R_i})^2\exp{\left\{\left(t_{i,\e}\phi_{R_{i-1},p_{i-1}}m_{l_i,R_i}\right)^2+\left|t_{i,\e}\phi_{R_{i-1},p_{i-1}}m_{l_i,R_i}\right|^{1+\e}\right\}}dx\\
&\ge \lambda \int_{B(p_{i-1},l_i)} m_{l_i,R_i}^2\exp{\left\{\left(t_{i,\e}m_{l_i,R_i}\right)^{2}\right\}}dx\\
&= \frac{\lambda}{2} \log{\frac{R_i}{l_i}}\exp{\left\{\frac{t_{i,\e}^{2}}{2\pi}\left(\log{\frac{1}{l_i}}-\log{\frac{1}{R_i}}\right)-2\log{\frac{1}{l_i}}-2\log{\frac{1}{1-(p_{i-1}/l_i)^2}}\right\}}\\
&\ge C \exp{\left(\delta'\log{\frac{1}{l_i}}\right)}, 
\end{split}
\]
for some constants $C,\delta'>0$ if $\e$ is small enough. Taking $\e\to0$, we get a contradiction. Then, analogously with the conclusion for (I), we obtain
\[
\limsup_{\e\to 0}I_{\e}(t_{i,\e}w_{i,\e})=\limsup_{\e\to 0}\frac{\|t_{i,\e}w_{i,\e}\|^2}{2}\le 2\pi.
\]
This proves (II).\\ \ \\
(III) We claim that $t_{k+1,\e}$ is bounded. To see this, we follow the argument on p493--494 in \cite{AY1}. We assume on the contrary, for a sequence $(\e_n)$, we have $\e_n\to0$ and $t_{k+1,\e_n}\to \infty$ as $n\to \infty$. Then we let 
\[
v_{n}:=\frac{t_{k+1,\e_n}w_{k+1,\e_n}}{\|t_{k+1,\e_n}w_{k+1,\e_n}\|}=\frac{w_{k+1,\e_n}}{\|w_{k+1,\e_n}\|}.
\]
Then using \eqref{eq:parameters}, we get $v_n\to v_0=u_0/\|u_0\|\not=0$ in $H^1_0(B)$. Furthermore, noting $t_{i,\e_n}$ is bounded for all $i=1,2,\cdots,k$ as proved in (I) and (II), we obtain
\[
\|w_{\e_n}\|^2=\sum_{i=1}^{k}t_{i,\e_n}^2+t_{k+1,\e_n}^2\|w_{k+1,\e_n}\|^2=t_{k+1,\e_n}^2\|w_{k+1,\e_n}\|^2(1+\eta_n),
\]
for a sequence $(\eta_n)\subset \R^+$ with $\eta_n\to 0$ as $n\to \infty$. Therefore, we get
\[
\frac{w_{\e_n}}{\|w_{\e_n}\|}=\frac{1}{(1+\eta_n)^{\frac12}}\left(v_n+\sum_{i=1}^k\frac{t_{i,\e_n}}{t_{k+1,\e_n}\|w_{k+1,\e_n}\|}w_{i,\e_n}\right)\to v_0\not=0\text{ in }H^1_0(B).
\]
Finally using $w_{\e_n}\in \mathcal{N}_{\e_n}$ and the Fatou lemma, we have
\[
\begin{split}
1&=\liminf_{n\to \infty}\frac{1}{\|w_{\e_n}\|^2}\int_Bf_{\e_n}(w_n)w_ndx\\&\ge\int_B\liminf_{n\to \infty}\frac{f_{\e_n}(w_{\e_n})}{w_{\e_n}}\left(\frac{w_{\e_n}}{\|w_{\e_n}\|}\right)^2dx\\
&=\infty,
\end{split}
\]
a contradiction. This proves the claim. Finally let us end the proof. We suppose the conclusion of (III) does not hold on the contrary. Then, we have a sequence $(\e_n)$ and a constant $\delta> 0$ such that $\e_n\to0$ as $n\to \infty$ and $I_{\e_n}(t_{k+1,\e_n}w_{k+1,\e_n})\ge I_0(u_0)+\delta$ for all $n$. On the other hand, as $t_{k+1,\e_n}$ is bounded, there exists a constant $t_0\ge0$ such that $t_{k+1,\e_n}\to t_0$ as $n\to \infty$ up to subsequences. This implies $t_{k+1,\e_n}w_{k+1,\e_n}\to t_0 u_0$ in $H^1_0(\Omega)$ as $n\to \infty$ and then, we get $t_0u_0\in\mathcal{N}_0$. It follows that $t_0=0$ or $1$. (See Step 2 in the proof of Lemma 3.4 in \cite{A1}.) Consequently, we deduce
\[
\lim_{n\to \infty}I_{\e_n}(t_{k+1,\e_n}w_{k+1,\e_n})\le I_0(u_0),
\] 
which implies a contradiction. This completes (III).
\end{proof}
\begin{lemma}\label{lem:bdd} There exist constants $0<K<K'$ such that
\[
K\le \|u\|^2\le K'
\]
for all $u\in \mathcal{N}_\e$ and small $\e>0$. 
\end{lemma}
\begin{proof} The lower bound is clearly confirmed by Lemma \ref{lem:energylb0} in Appendix \ref{sec:basic}. On the other hand, the upper bound is proved similarly to claim 1 on p404 in \cite{A1}. This finishes the proof. 
\end{proof}
%
%
%
%

%
%
Next we study the behavior of $r_{i,\e}$. To this end we recall the next lemma. 
\begin{lemma}[Radial lemma \cite{S}]\label{lem:rl}
Let $B^N\subset \R$ be a $N$-dimensional unit ball and $H_{\text{rad}}(B^N)$ be a subspace of $H^1(B^N)$ which consists of all the radial functions. Then, there exists a constant $c_N>0$ such that 
\[
|u(r)|\le c_N\|u\|/r^{\frac{N-1}{2}}\ (u\in H_{\text{rad}}(B^N)\text{ and }r\in (0,1)).
\]
In particular, for $N=2$ we have $|u(r)|\le c_2 \|u\|/\sqrt r$.
\end{lemma}
We deduce the following. 
\begin{lemma}\label{lem:kthzero} We see
\[
r_{i,\e}\to0\text{ as }\e\to0
\]
for all $i=1,2,\cdots,k$.
\end{lemma}
\begin{proof}
By Lemma \ref{lem:bdd}, we may assume $u_{\e}$ is bounded in $H^1_0(B)$ and $u_{\e}\rightharpoonup u$ weakly in $H^1_0(B)$ as $\e\to0$ where $u$ is a radial solution $u$ to \eqref{i11} with $\e=0$. Moreover we recall that $u_{i,\e}=u_{\e}|_{B(r_{i-1,\e},r_{i,\e})}$  satisfies $(-1)^{i-1}u_{i,\e}\ge 0$ for all $i=1,2,\cdots,k+1$. Then, we can suppose there exists a function $u_i\in H^1_0(B)$ such that $u_{i,\e}\rightharpoonup u_{i}$ weakly in $H^1_0(B)$ and $(-1)^{i-1}u_i\ge 0$ for all $i=1,2,\cdots,k+1$ and further, $u=\sum_{i=1}^{k+1} u_i$. Now, let us show $r_{k,\e}\to 0$ which also implies $r_{i,\e}\to 0$ for all $i=1,2,\cdots,k-1$ as $\e\to0$. If not, we may suppose that there exists a constant $r_k\in(0,1]$ such that $r_{k,\e}\to r_k$ as $\e\to0$.  We then claim $u_{k+1}\not=0$. Indeed, if $u_{k+1}=0$, on the contrary, we have $\int_{B}u_{k+1,\e}^2dx\to0$ as $\e\to0$. It follows that $\|u_{k+1,\e}\|_{\infty}=\sup_{r\in (r_{k,\e},1)}u_{k+1,\e}(r)\to \infty$ as $\e\to0$. Otherwise, from Lemma \ref{lem:bdd}, we get
\[
\begin{split}
0<K&\le \|u_{k+1,\e}\|^2 =\lambda\int_{B}u_{k+1,\e}^2e^{u_{k+1,\e}^{2}+|u_{k+1,\e}|^{1+\e}}dx\\&\le \lambda e^{\|u_{k+1,\e}\|_{\infty}^2+\|u_{k+1,\e}\|_{\infty}^{1+\e}}\int_{B} u_{k+1,\e}^2dx\to 0
\end{split}
\]
as $\e\to0$, a contradiction. As a consequence, setting $\|u_{k+1,\e}\|_{\infty}=u_{k,\e}(r_{k,\e}^*)$ with a value $r_{k,\e}^*\in(r_{k,\e},1)$, we get from Lemma \ref{lem:rl} that
\[
\|u_{k,\e}\|\ge c_2^{-1}|u_{k,\e}(r^*_{k,\e})|{(r^*_{k,\e})}^{\frac{1}{2}}\ge c_2^{-1}|u_{k,\e}(r^*_{k,\e})|r_{k,\e}^{\frac{1}{2}}\to \infty
\] 
as $\e\to0$ since $r_k>0$, which contradicts Lemma \ref{lem:bdd}. This shows the claim. Especially we get $0\le r_1\le \cdots\le r_k\in(0,1)$. Now recalling that $u$ is a radial solution and $\lambda<\lambda_{\text{AY}}$ and then, noting $(-1)^k u_{k+1}\ge 0$ is nontrivial and $(-1)^{k-1}u_{k}\ge0$, we must have $u_{k}=0$. Then, the maximum principle yields $r_{k}=r_{k+1}$. Finally, repeating the argument above, we get $\sup_{r\in (r_{k-1,\e},r_{k,\e})}u_{k,\e}(r)\to \infty$ as $\e\to0$ and then Lemmas \ref{lem:bdd} and  \ref{lem:rl} lead us to the contradiction. This finishes the proof. 
\end{proof}
Finally, let us investigate the limit value of the energy $I_\e(u_\e)$ more precisely. 
\begin{lemma}\label{lem:energy3}
We get
\[
\lim_{\e\to0}I_\e(u_{i,\e})= 2\pi
\]
for all $i=1,2,\cdots,k$. Furthermore,  we obtain
\[
\limsup_{\e\to0} I_{\e}(u_{k+1,\e})=I_0(u_0).
\]
\end{lemma}
\begin{proof} Choose $i=1,2,\cdots,k$. We first claim
\begin{equation}
\liminf_{\e\to0}I_\e(u_{i,\e})\ge 2\pi.\label{eq:claim1}
\end{equation}
Indeed, let $\tilde{u}_{0,\e}\in H^1_0(B)$ be a positive solution of \eqref{i11} with $B$ replaced by $B_{r_{i,\e}}$ which satisfies
\[
I_{\e}(\tilde{u}_{0,\e})=\inf\left\{I_{\e}(u)\ \Big|\ u\in H^1_0(B_{r_{i,\e}}),\ \int_{B_{r_{i,\e}}}|\nabla u|^2dx=\int_{B_{r_{i,\e}}}f_\e(u)udx.\right\}.
\]
The existence of $\tilde{u}_{0,\e}$ is ensured by \cite{A1}. Then we have $I_{\e}(u_{i,\e})\ge I_{\e}(\tilde{u}_{0,\e})$. Hence it suffices to show  $\liminf_{\e\to0}I_{\e}(\tilde{u}_{0,\e})\ge 2\pi$. Now we assume, on the contrary, $\liminf_{\e\to0}I_{\e}(\tilde{u}_{0,\e})< 2\pi$. Set $v_{i,\e}(x)=\tilde{u}_{0,\e}(r_{i,\e}x)$. Then $v=v_{i,\e}$ satisfies
\begin{equation}
\begin{cases}
-\Delta v=\lambda r_{i,\e}^2 ve^{v^{2}+v^{1+\e}},\ v>0\ \text{ in }B,\\
v=0\text{ on }\partial B.
\end{cases}\label{eq:e3-1}
\end{equation}
We define the energy associated to \eqref{eq:e3-1}.
\[
J_{\e}(v)=\int_{B}|\nabla v|^2dx- r_{i,\e}^2\int_{B}F_{\e}(v)dx\ (v\in H^1_0(B)).
\]
Then we have $I_{\e}(\tilde{u}_{0,\e})=J_{\e}(v_{i,\e})$ and thus, $\liminf_{\e\to0}J_{\e}(v_{i,\e})< 2\pi$. In particular, we have a sequence $(\e_n)$ such that $\e_n\to0$ as $n\to \infty$ and $c:=\lim_{n\to \infty}J_{\e_n}(v_{i,\e_n})<2\pi$. Notice that Lemma \ref{lem:bdd} ensures $c>0$. Then, noting $J_{\e_n}'(v_{i,\e_n})=0$ and Lemma \ref{lem:cpt} in Appendix \ref{sec:basic}, we can find a function $v_0\in H^1_0(\Omega)$ such that $v_{i,\e_n}\to v_0$ in $H^1_0(\Omega)$ as $n\to \infty$ up to subsequences. Lastly, using \eqref{eq:e3-1}, we get 
\[
\begin{cases}
-\Delta v_0=0,\ v_0\ge0\ \text{ in }B,\\
v_0=0\text{ on }\partial B.
\end{cases}
\]
Then, the maximum principle shows $v_0=0$. But this contradicts $c>0$. Next let us show
\begin{equation}
\limsup_{\e\to0}I_\e(u_{i,\e})\le 2\pi,\text{\ \ and\ \ }\limsup_{\e\to0} I_{\e}(u_{k+1,\e})=I_0(u_0).\label{eq:claim2}
\end{equation}
In fact, we get by Lemma \ref{lem:energyup} and \eqref{eq:claim1} that
\[
2\pi k +I_0(u_0)\ge \limsup_{\e\to0}I_{\e}(u_{\e})\ge 2\pi k+\limsup_{\e\to 0}I_{\e}(u_{k+1,\e})
\]
which implies $I_0(u_0)\ge \limsup_{\e\to0}I_{\e}(u_{k+1,\e})$. Furthermore, let $u_{0,\e}$ be the least energy solution of \eqref{i11} obtained by \cite{A1}. It follows that $I_0(u_0)\ge \limsup_{\e\to0} I_{\e}(u_{k+1,\e})\ge \limsup_{\e\to0}I_{\e}(u_{0,\e})$. We claim $\limsup_{\e\to0}I_{\e}(u_{0,\e})\ge I_0(u_0)$. If not, we have a sequence $(\e_n)$  such that $\e_n\to 0$ as $n\to \infty$ and $\lim_{n\to \infty}I_{\e_n}(u_{0,\e})<I_0(u_0)$. Note $I_0(u_0)\in(0,2\pi)$. Then from Lemma \ref{lem:cpt}, we deduce, by subtracting a subsequence if necessary, $u_{0,\e_n}\to \tilde{u}_0$ in $H^1_0(B)$ as $n\to \infty$ and further, $\tilde{u}_0$ is a nontrivial solution of \eqref{i9} with $I_0(\tilde{u}_0)\in(0,I_0(u_0))$. But as $\tilde{u}_{0}\in\mathcal{N}_{0}$, we obtain a contradiction by the definition of $u_0$. This proves the claim. Now again arguing as the beginning, we get
\[
2\pi k +I_0(u_0)\ge \limsup_{\e\to 0}I_{\e}(u_{\e})\ge 2\pi(k-1)+\limsup_{\e\to0}I_{\e}(u_{i,\e})+I_0(u_0).
\]
This completes \eqref{eq:claim2}. As a consequence, \eqref{eq:claim1} and \eqref{eq:claim2} finish the proof.
\end{proof}
\begin{lemma}\label{lem:energy5}
We have
\[
\lim_{\e\to0}I_\e(u_{k+1,\e})= I_0(u_0).
\]
\end{lemma}
\begin{proof}
Since $\liminf_{\e\to0} I_{\e}(u_{k+1,\e})\le I_0(u_0)$, arguing as in the previous proof, we can get
\[
\liminf_{\e\to0} I_{\e}(u_{k+1,\e})= I_0(u_0).
\]
Then combining this together with the final assertion in the previous lemma, we complete the proof.
\end{proof}
\section{Behavior of $u_\e$ in the ball $B_{r_{1,\e}}$}\label{sec:ball}
Let us start our main argument with studying the behavior on a ball. To this end, we first observe that $u_{1,\e}=u_{\e}|_{B_{r_{1,\e}}}$ is a solution to  
\begin{equation}
\begin{cases}\label{a1}
-\Delta u=\lambda ue^{u^{2}+|u|^{1+\e}},\ u>0\text{ in }B_{r_{1,\e}}\\
u=0\text{ on }\partial B_{r_{1,\e}},
\end{cases}
\end{equation}
for $\e>0$. Then the results in \cite{GNN} shows that $u_{1,\e}$  is radial and $\|u_{1,\e}\|_{L^{\infty}(B_{r_{1,\e}})}=u_{1,\e}(0)$. Next we see that $v_{1,\e}(x):=u_{\e}(r_{1,\e}x)$ ($x\in B_1$) is a solution of
\begin{equation}
\begin{cases}\label{a2}
-\Delta v=\lambda r_{1,\e}^2ve^{v^{2}+|v|^{1+\e}},\ v>0\text{ in }B\\
v=0\text{ on }\partial B,
\end{cases}
\end{equation}
for $\e>0$ and $\|v_{1,\e}\|_{L^{\infty}(B)}=v_{\e}(0)$. Notice $\lambda r_{1,\e}^2\to 0$ as $\e\to0$ by Lemma \ref{lem:kthzero}. Furthermore, by Lemma \ref{lem:energy3}, 
we get
\[
J_\e(v_\e):=\frac12\int_B|\nabla v_\e|^2dx-r_{1,\e}^2\int_BF_\e(v_\e)dx\to 2\pi,
\]
as $\e\to0$. We have the following
\begin{proposition}\label{prop:ball} We get $v_{1,\e}\rightharpoonup 0$ weakly in $H^1_0(B)$, $v_{1,\e}(0)\to \infty$ and 
\[
\int_B|\nabla v_{1,\e}|^2dx\rightarrow 4\pi,
\]
as $\e\to0$. Furthermore, let $\gamma_{1,\e}>0$ be such that $2\lambda r_{1,\e}^2v_{1,\e}(0)^2e^{v_{1,\e}(0)^2+v_{1,\e}(0)^{1+\e}}\gamma_{1,\e}^2=1$. Then we have $\gamma_{1,\e}\to 0$ and
\[
2v_{1,\e}(0)(v_{1,\e}(\gamma_{1,\e} x)-v_{1,\e}(0))\to \log{\frac{1}{(1+|x|^2/8)^2}}\text{ in }C^2_{\text{loc}}(\R^2),
\]
as $\e\to0$.
\end{proposition}
\begin{proof}
It is a direct consequence of Theorem 2 in \cite{AD}. 
\end{proof}
\begin{corollary}\label{cor:ball}
We obtain $u_{1,\e}\rightharpoonup 0$ weakly in $H^1_0(B)$, $u_{1,\e}(0)\to\infty$ and
\[
\int_{B_{r_{1,\e}}}|\nabla u_{1,\e}|^2dx\rightarrow 4\pi,
\]
and
\[
I_{\e}(u_{1,\e})\to 2\pi,
\]
as $\e\to0$. Furthermore, let $\delta_{1,\e}=r_{1,\e}\gamma_{1,\e}>0$. Then we have $\delta_{1,\e}\to 0$ and
\[
2u_{\e}(0)(u_{\e}(\delta_{1,\e} x)-u_{\e}(0))\to \log{\frac{1}{(1+|x|^2/8)^2}}\text{ in }C^2_{\text{loc}}(\R^2),
\]
as $\e\to0$.
\end{corollary}
\begin{proof} The proof follows from Proposition \ref{prop:ball} and Lemma \ref{lem:energy3}.
\end{proof}
%
%
%
%
%
\section{Behavior of $u_\e$ on annuli}\label{sec:annulus}
We next investigate the behavior of $u_\e$ on annuli. Fix $i\in\{2,\cdots,k\}$ and set $u_{i,\e}:=u_{\e}|_{B(r_{i-1,\e},r_{i,\e})}$. Then $u_{i,\e}\in H^1_0(B)$ by zero extension. Since $u_{i,\e}$ is radial, we may assume it satisfies 
\begin{equation}\label{c1}
\begin{cases}
-u_{i,\e}''-\frac1ru_{i,\e}'=\lambda u_{i,\e} e^{u_{i,\e}^{2}+u_{i,\e}^{1+\e}},&\hbox{ in }(r_{i-1,\e},r_{i,\e})\\
u_{i,\e}>0&\hbox{ in }(r_{i-1,\e},r_{i,\e})\\
u_{i,\e}(r_{i-1,\e})=u_{i,\e}(r_{i,\e})=0.
\end{cases}
\end{equation}
Now we have the following result.
\begin{proposition}\label{prop:annulus}
We get $u_{i,\e}\rightharpoonup 0$ weakly in $H^1_0(B)$,
\[
\int_{B(r_{i-1,\e},r_{i,\e})}|\nabla u_{i,\e}|^2dx\rightarrow 4\pi,
\]
and
\[
I_{\e}(u_{i,\e})\rightarrow 2\pi,
\]
as $\e\to0$. Moreover, let us denote by $M_{i,\e} r_{i,\e}\in(r_{i-1,\e},r_{i,\e})$ with $M_{i,\e}<1$, the point such that $||u_{i,\e}||_{L^{\infty}(r_{i-1,\e},r_{i,\e})}=u_{i,\e}(M_{i,\e}r_{i,\e})$. Then if we set $\delta_{i,\e}=\gamma_{i,\e}r_{i,\e}>0$ with 
\[
2\lambda||u_{i,\e}||_{L^{\infty}(r_{i-1,\e},r_{i,\e})}^2e^{||u_{i,\e}||_{L^{\infty}(r_{i-1,\e},r_{i,\e})}^2+||u_{i,\e}||_{L^{\infty}(r_{i-1,\e},r_{i,\e})}^{1+\e}}r_{i,\e}^2\gamma_{i,\e}^2=1,
\]
we get $\delta_{i,\e}\to 0$ and further, 
\[
\begin{split}
2||u_{i,\e}||_{L^{\infty}(r_{i-1,\e},r_{i,\e})}(u_{i,\e}(M_\e r_{i,\e}+\delta_{i,\e} r)&-||u_{i,\e}||_{L^{\infty}(r_{i-1,\e},r_{i,\e})})\\&\rightarrow \log\frac{1}{\left(1+r^2/8\right)^2} \text{ in } C^2_{\text{loc}}(\R^+),
\end{split}
\]
as $\e\to0$.
\end{proposition} 
In the following, we set $M_\e:=M_{i,\e}$ for simplicity. We get the following.
\begin{lemma}\label{lem:infty}
$u_{i,\e}(M_\e r_{i,\e})\rightarrow+\infty$ as $\e\to0$.
\end{lemma}
\begin{proof}
Integrating \eqref{c1} we get
\[
\begin{split}
&\int_{r_{i-1,\e}}^{r_{i,\e}}(u_{i,\e}')^2rdr=\lambda\int_{r_{i-1,\e}}^{r_{i,\e}}u_{i,\e}^2e^{u_{i,\e}^2+u_{i,\e}^{1+\e}}rdr
\\
&\le\lambda e^{u_{i,\e}^2(M_\e r_{i,\e})+u_{i,\e}^{1+\e}(M_\e r_{i,\e})}\int_{r_{i-1,\e}}^{r_{i,\e}}u_{i,\e}^2rdr\\
&\hbox{(using the Poincare inequality) }
\le \lambda \frac{e^{u_{i,\e}^2(M_\e r_{i,\e})+u_{i,\e}^{1+\e}(M_\e r_{i,\e})}}{\lambda_1(r_{i-1,\e},r_{i,\e})}\int_{r_{i-1,\e}}^{r_{i,\e}}(u_{i,\e}')^2rdr,
\end{split}
\]
where  $\lambda_1(r_{i-1,\e},r_{i,\e})$ is the first eigenvalue of the operator $-u''-\frac1ru'$ in $(r_{i-1,\e},r_{i,\e})$. Since $r_{i-1,\e},r_{i,\e}\rightarrow0$ we get that $\lambda_1(r_{i-1,\e},r_{i,\e})\rightarrow+\infty$ as $\e\rightarrow0$. This gives the claim.
\end{proof}
Now, let us consider the scaled function, $v_\e:\left(\frac{r_{i-1,\e}}{r_{i,\e}},1\right)\rightarrow\R$ defined as
\[v_\e(r)=u_{i,\e}(r_{i,\e}r)
\]
which satisfies
\begin{equation}\label{c2}
\begin{cases}
-v_\e''-\frac1rv_\e'=\lambda r_{i,\e}^2v_\e e^{v_\e^{2}+v_\e^{1+\e}}&\hbox{ in }\left(\frac{r_{i-1,\e}}{r_{i,\e}},1\right),\\
v_\e>0&\hbox{ in }\left(\frac{r_{i-1,\e}}{r_{i,\e}},1\right),\\
v_\e(\frac{r_{i-1,\e}}{r_{i,\e}})=v_\e(1)=0.
\end{cases}
\end{equation}
Set
\[
r_\e=\frac{r_{i-1,\e}}{r_{i,\e}}.
\]
Then we have the following local behavior.
\begin{lemma}\label{lem:4-step3}
Choose $M_\e \in(r_{\e},1)$ as in Proposition \ref{prop:annulus}. Then if we set $\gamma_{i,\e}>0$ so that
\[
2\lambda\|v_{\e}\|_{L^{\infty}(r_{\e},1)}^2e^{\|v_{\e}\|_{L^{\infty}(r_{\e},1)}^2+\|v_{\e}\|_{L^{\infty}(r_{\e},1)}^{1+\e}}r_{i,\e}^2\gamma_{i,\e}^2=1,
\]
we get $\gamma_{i,\e}\to 0$ and 
\[
2\|v_{\e}\|_{L^{\infty}(r_{\e},1)}(v_\e(M_\e+\gamma_{i,\e}r)-\|v_{\e}\|_{L^{\infty}(r_{\e},1)})\rightarrow z(r)=\log\frac{1}{\left(1+r^2/8\right)^2} \text{ in } C^2_{\text{loc}}(0,+\infty),
\]
as $\e\to0$.
\end{lemma}
\begin{proof}
Let $v_\e$, $r_{\e}$ and $M_\e\in(r_\e,1)$ as above. For $\gamma_\e>0$, which will be chosen later,  we define the scaled function
\begin{equation}
z_\e(r)=2v_{\e}(M_{\e})(v_\e(M_\e+\gamma_\e r)-v_{\e}(M_{\e})).
\end{equation}
We have that $z_\e$ solves the equation,
\[
\begin{cases}
-z_\e''-\frac1{\frac{M_\e}{\gamma_\e}+r}z_\e'=2\lambda\gamma_\e^2 r_{i,\e}^2e^{ v_\e^2(M_\e)+ v_\e^{1+\e}(M_\e)} v_\e^2(M_\e)\left(\frac{z_\e}{2v_{\e}^2(M_\e)}+1\right)\\\ \ \ \ \ \ \ \ \ \ \ \ \ \ \ \ \times e^{\left\{z_\e\left(\frac{z_\e}{4v_{\e}^2(M_\e)}+1\right)+v_\e^{1+\e}(M_\e)\left(\left|\frac{z_\e}{2v_{\e}^2(M_\e)}+1\right|^{1+\e}-1\right)\right\}}
\hbox{ in }\left(\frac{r_\e-M_\e}{\gamma_\e},\frac{1-M_\e}{\gamma_\e}\right),\\
z_\e(r)\le0,\ z_\e(0)=z_\e'(0)=0,\\
z_\e\left(\frac{r_\e-M_\e}{\gamma_\e}\right)=z_\e\left(\frac{1-M_\e}{\gamma_\e}\right)=-2v_\e^2(M_\e)\rightarrow-\infty\ (\e\to0).
\end{cases}
\]
So setting
\[
2\lambda\gamma_\e^2 r_{i,\e}^2e^{ v_\e^2(M_\e)+ v_\e^{1+\e}(M_\e)} v_\e^2(M_\e)=1
\]
we get
\begin{equation}\label{b1}
\begin{cases}
-z_\e''-\frac1{\frac{M_\e}{\gamma_\e}+r}z_\e'=\left(\frac{z_\e}{2v_{\e}^2(M_\e)}+1\right)e^{\left\{z_\e\left(\frac{z_\e}{4v_{\e}^2(M_\e)}+1\right)+v_\e^{1+\e}(M_\e)\left(\left|\frac{z_\e}{2v_{\e}^2(M_\e)}+1\right|^{1+\e}-1\right)\right\}}\\
\ \ \ \ \ \ \ \ \ \ \ \ \ \ \ \ \ \ \ \ \ \ \ \ \ \ \ \ \ \ \ \ \ \ \ \ 
\ \ \ \ \ \ \ \ \ \ \ \ \ \ \ \ \ \ \ \ \ \ \ \ \ \ \ \ \ \ \ \ \ \ \ \ 
\hbox{ in }\left(\frac{r_\e-M_\e}{\gamma_\e},\frac{1-M_\e}{\gamma_\e}\right),\\
z_\e(r)\le0,\ z_\e(0)=z_\e'(0)=0,\\
z_\e\left(\frac{r_\e-M_\e}{\gamma_\e}\right)=z_\e\left(\frac{1-M_\e}{\gamma_\e}\right)=-2v_\e^2(M_\e)\rightarrow-\infty \ (\e\to0).
\end{cases}
\end{equation}
Note that $\gamma_\e\to0$ as $\e\to0$. Actually, multiplying \eqref{c2} by $v_\e r$ and integrating over $(0,1)$, we get 
\[
\begin{split}
&\int_{0}^1(v_\e')^2rdr=\lambda r_{i,\e}^2\int_0^1v_\e^2e^{v_\e^2+v_\e^{1+\e}}rdr\\
&\le\lambda r_{i,\e}^2e^{v_\e^2(M_\e)+v_\e^{1+\e}(M_\e)}\int_0^1v_\e^2rdr\\
&\hbox{(applying the Poincare inequality )}\le  \frac{\lambda}{\lambda_1}r_{i,\e}^2e^{v_\e^2(M_\e)+v_\e^{1+\e}(M_\e)}\int_{0}^1(v_\e')^2rdr.
\end{split}
\]
This shows 
\[r_{i,\e}^2e^{v_\e^2(M_\e)+v_\e^{1+\e}(M_\e)} \ge C>0
\]
for some constant $C>0$ and small $\e>0$ . Then noting our choice of $\gamma_\e$ and  Lemma \ref{lem:infty}, we prove the claim. Moreover we clearly have that $\lim_{\e\to0}\frac{1-M_\e}{\gamma_\e}\rightarrow \infty$, $\lim_{\e\to0}\frac{M_\e-r_\e}{\gamma_\e}=l\in[0,\infty]$ and $\lim_{\e\to0}\frac{M_\e}{\gamma_\e}=m\in [l,\infty]$. Now let us show that for any compact subset $K\subset\subset (-l,\infty)$ ($[0,\infty)$ if $l=0$), there exists a constant $C>0$ which is independent of $\e$ such that
\[
\|z_\e\|_{C^1(K)}\le C. 
\]
Indeed, from \eqref{b1}, we get that $-z_\e''-\frac1{\frac{M_\e}{\gamma_\e}+r}z_\e'\le1$. First assume $l>0$ and choose any $K\subset\subset (-l,0]$. We may suppose  $K\subset\subset (\frac{r_\e-M_\e}{\gamma_\e},0]$ for small $\e>0$. Define $a=\min K< 0$ and set $C_\e=\frac{M_\e}{\gamma_\e}$. Then, for any $r\in K$, we derive,
\[
-\left[z_\e'(r)(C_\e+r)\right]'\le C_\e+r.
\]
Integrating between $r$ and $0$ we obtain
\[
z_\e'(r)(C_\e+r)\le-\left(C_\e r+\frac12r^2\right)
\]
Since $C_\e+r>0$ for small $\e>0$, we show
\[
z_\e'(r)\le-\frac{C_\e r+\frac12r^2}{C_\e+r}\text{ and thus, }
z_\e(r)\ge\int_r^0\frac{C_\e s+\frac12s^2}{C_\e+s}ds
\]
for small $\e>0$. If we set $G_\e(s)=\frac{C_\e s+\frac12s^2}{C_\e+s}$, we get that $G_\e'(s)\ge0$ for all $s\in K$. So we find that $G_\e(s)\ge G_\e(a)$ for all $s\in K$. Now, if $C_\e\to \infty$ as $\e\to 0$, we get  $G_\e(a)\ge -2|a|$ for small $\e>0$. If $C_\e$ is bounded, we get a constant $c_0>0$ such that $G_\e(a)\ge -c_0$ for small $\e>0$. This implies that there exists a constant $c_1>0$ such that
\begin{equation}
z_\e'(r)\le c_1\text{ and thus,\  }z_\e\ge c_1a \text{ on }K,
\end{equation}
for all small $\e$. Hence we have a constant $C>0$ such that $\|z_\e\|_{C^1(K)}\le C$  uniformly for small $\e>0$. On the other hand, for any compact subset $K\subset\subset [0,\infty)$, repeating the same argument as above,  we get the desired uniform bound for $\|z_\e\|_{C^1(K)}$. This proves the claim. Consequently, we may pass to the limit in the equation \eqref{b1}. Now let us discuss the "limit domain". We have three possibilities,
\begin{itemize}
\item[$1.$]$\frac{r_\e-M_\e}{\gamma_\e}\rightarrow-\infty$,
\item[$2.$]$\frac{r_\e-M_\e}{\gamma_\e}\rightarrow -l<0$.
\item[$3.$]$\frac{r_\e-M_\e}{\gamma_\e}\rightarrow0$,
\end{itemize}
We will show that only case $3$ occurs.
\vskip0.2cm
{\bf Case 1: $\frac{r_\e-M_\e}{\gamma_\e}\rightarrow-\infty$ cannot occur}\\
First we note that in this case we have that  $\frac{M_\e}{\gamma_\e}\rightarrow+\infty$. Then, passing to the limit in \eqref{b1}, we get that there exists a function $z$ which satisfies $z_\e\to z$ in $C^2_{\text{loc}}(\R)$ and
\begin{equation}
\begin{cases}
-z''=e^z\qquad\hbox{ in }\R,\\
z(0)=z'(0)=0.
\end{cases}
\end{equation}
Hence $z(s)=\log\frac{4e^{\sqrt2s}}{\left(1+e^{\sqrt2s}\right)^2}$.  So we have that
\begin{equation*}\begin{split}
&\int_{r_\e}^1|v_\e'|^2rdr=\lambda r_{i,\e}^2\int_{r_\e}^1 v_\e^2 e^{v_\e^2+v_\e^{1+\e}}rdr\\
&=\lambda r_{i,\e}^2\gamma_\e e^{v_\e^2(M_\e)+v_\e^{1+\e}(M_\e)}v_\e^2(M_\e)\\
&\times\int_{\frac{r_\e-M_\e}{\gamma_\e}}^{\frac{1-M_\e}{\gamma_\e}}\left(\frac{z_\e(r)}{2v_\e^2(M_\e)}+1\right)^2 e^{z_\e(r)\left(\frac{z_\e(r)}{4v_\e^2(M_\e)}+1\right)+v_\e^{1+\e}(M_\e)\left(\left|\frac{z_\e(r)}{2v_\e^2(M_\e)}+1\right|^{1+\e}-1\right)}(M_\e+\gamma_\e r)dr\\
&\ge\lambda r_{i,\e}^2\gamma_\e M_\e e^{v_\e^2(M_\e)+v_\e^{1+\e}(M_\e)}v_\e^2(M_\e)\\&\times\int_{0}^{\frac{1-M_\e}{\gamma_\e}}\left(\frac{z_\e(r)}{2v_\e^2(M_\e)}+1\right)^2 e^{z_\e(r)\left(\frac{z_\e(r)}{4v_\e^2(M_\e)}+1\right)+v_\e^{1+\e}(M_\e)\left(\left|\frac{z_\e(r)}{2v_\e^2(M_\e)}+1\right|^{1+\e}-1\right)}dr\\
&=\frac{M_\e }{2\gamma_\e}
\int_{0}^{\frac{1-M_\e}{\gamma_\e}}\left(\frac{z_\e(r)}{2v_\e^2(M_\e)}+1\right)^2 e^{z_\e(r)\left(\frac{z_\e(r)}{4v_\e^2(M_\e)}+1\right)+v_\e^{1+\e}(M_\e)\left(\left|\frac{z_\e(r)}{2v_\e^2(M_\e)}+1\right|^{1+\e}-1\right)}dr.
\end{split}\end{equation*}
Here Fatou's lemma implies that 
\[
\begin{split}
\liminf_{\e\to0}\int_{0}^{\frac{1-M_\e}{\gamma_\e}}\left(\frac{z_\e(r)}{2v_\e(M_\e)}+1\right)^2 e^{z_\e(r)\left(\frac{z_\e(r)}{4v_\e^2(M_\e)}+1\right)+v_\e^{1+\e}(M_\e)\left(\left|\frac{z_\e(r)}{2v_\e^2(M_\e)}+1\right|^{1+\e}-1\right)}dr\\\ge\int_{0}^{+\infty}e^{z(s)}dr>0.
\end{split}
\]
Therefore by Lemma \ref{lem:bdd}, we deduce a contradiction since  $M_\e/\gamma_\e\to \infty$ as $\e\to0$. This ends Case 1.
\vskip0.2cm
{\bf Case 2: $\frac{r_\e-M_\e}{\gamma_\e}\rightarrow-l<0$ cannot occur}\\
Noting $m:=\lim\limits_{\e\rightarrow0}\frac{M_\e}{\gamma_\e}$ and $m\ge l$, we get, passing to the limit in \eqref{b1}, that the weak limit $z$ satisfies
\begin{equation*}
\begin{cases}
-z''-\frac1{m+r}z'=
e^{z}\qquad\hbox{ in }\left(-l,+\infty\right)\\
z(r)\le0,\ z(0)=z'(0)=0.
\end{cases}
\end{equation*}
Then, setting $Z(s)=z(s-m)$ we derive that $Z$ satisfies
\begin{equation*}
\begin{cases}
-Z''-\frac1rZ'=
e^Z\qquad\hbox{ in }\left(m-l,+\infty\right)\\
Z(r)\le0,\ Z(m)=Z'(m)=0.\\
\end{cases}
\end{equation*}
This Cauchy problem admits the unique solution (see \cite{GGP})
\begin{equation*}\label{b2}
Z(s)=\log\frac{4\alpha^2m^{\alpha+2}s^{\alpha-2}}
{\left((\alpha+2)m^\alpha+(\alpha-2)s^\alpha\right)^2}
\end{equation*}
where $\alpha=\sqrt{2m^2+4}$. 
Let us show $m=l$. To this end, we can proceed as in Lemma 3.5 in \cite{GGP}. For the sake of the completeness, we sketch it. We shall show that $z_{\e}((r_{\e}-M_{\e})/\gamma_\e)\to -\infty$ implies that $m=l$. Indeed, arguing as above, we have that for any $r\in [(r_\e-M_\e)/\gamma_\e,0]$,
$$z_\e'(r)\left(\frac{M_\e}{\gamma_\e}+r\right)\le-\left(\frac{M_\e}{\gamma_\e} r+\frac12r^2\right).$$
If by contradiction we have that $m>l$, we deduce that $\frac{M_\e}{\gamma_\e}+r\ge m-l+o(1)$ where $o(1)\to 0$ as $\e\to0$ and then we get that
$$z_\e'(r)\le C\hbox{ in }[(r_\e-M_\e)/\gamma_\e,0]$$
for a constant $C>0$ which is independent of small $\e>0$. On the other hand, by the mean value theorem, since $z_{\e}((r_{\e}-M_{\e})/\gamma_\e)\to -\infty$ and $z_{\e}(0)=0$ we deduce the existence of $\xi_\e\in\left(\frac{r_{\e}-M_{\e}}{\gamma_\e},0\right)$ such that $z_\e'(\xi_\e)\to -\infty$ which gives a contradiction. So $m=l$. Now, from Lemmas \ref{lem:energy3}, \ref{lem:F} and the blow up procedure as above, we get
\[
\begin{split}
2&= \lambda r_{i,\e}^2\int_{r_\e}^{1}v_{\e}^2e^{v_{\e}^{2}+v_{\e}^{1+\e}}rdr+o(1)\\
 &= \frac{1}{2}\int_{\frac{r_\e-M_\e}{\gamma_\e}}^{\frac{1-M_\e}{\gamma_\e}}\left(\frac{z_\e(r)}{2v_\e^2(M_\e)}+1\right)^2 \times\\
&\ \ \ \ \ \ \ \ \ \ \ \ e^{z_\e(r)\left(\frac{z_\e(r)}{4v_\e^2(M_\e)}+1\right)+v_\e^{1+\e}(M_\e)\left(\left|\frac{z_\e(r)}{2v_\e^2(M_\e)}+1\right|^{1+\e}-1\right)}\left(\frac{M_{\e}}{\gamma_\e}+r\right)dr+o(1),
\end{split}
\]
where $o(1)\to0$ as $\e\to0$. Then using $m=l>0$ and Fatou's Lemma, we obtain
\[
2\ge \frac{1}{2}\int_0^{\infty}e^{Z(s)}sds=\sqrt{2m^2+4}>2,
\]
a contradiction. This finishes Case 2.
\vskip0.2cm

\textbf{Case 3: $\frac{r_{\e}-M_{\e}}{\gamma_\e}\to 0$ occurs.}\\
Repeating the procedure in Case 2 we can show $m=l=0$. As a consequence, we deduce 
\[
z_\e\to z\text{ in }C_{\text{loc}}([0,\infty))\cap C_{\text{loc}}^2((0,\infty))
\]
and then, $z$ satisfies 
\begin{equation*}
\begin{cases}
-z''-\frac1{r}z'=
e^{z}\qquad\hbox{ in }\left(0,+\infty\right)\\
z(r)\le0,\ z(0)=0.
\end{cases}
\end{equation*}
The previous equation can be integrate giving the solutions (see\cite{GGP}, p. 744-745)
\begin{equation}\label{C}
z(r)=\log\left(\frac4{\delta^2}\frac{e^{\sqrt2\frac{\log r-y}\delta}}{\left(1+e^{\sqrt2\frac{\log r-y}\delta}\right)^2}
\right)-2\log r
\end{equation}
for some constants $\delta\ne0,y\in\R$. Moreover a direct calculation shows
\[
z(r)=2\log{\frac{2}{\delta}}-\frac{\sqrt{2}}{\delta}y+\left(\frac{\sqrt{2}}{\delta}-2\right)\log{r}-2\log{\left(1+e^{\sqrt2\frac{\log r-y}\delta}\right)}.
\]
Since $z(0)=0$, we must have $\delta=1/\sqrt{2}$. Then we clearly deduce $y=\log{2\sqrt{2}}$. 
This completes the proof.
\end{proof}
\begin{proof}[Proof of Proposition \ref{prop:annulus}] The proposition follows from Lemmas \ref{lem:energy3}, \ref{lem:F} and \ref{lem:4-step3}. 
\end{proof}
%
%
%
%
\begin{remark}\label{R}
If we consider a radial nodal solution $u_p$ to the problem
\begin{equation}
\begin{cases}
-\Delta u=|u|^{p-1}u&\text{ in }B\\
u=0&\text{ on }\partial B,
\end{cases}
\end{equation}
then in Proposition 3.1 of \cite{GGP} it was proved that case $2$ occurs for some suitable $m<0$. This shows that the shape of the nonlinearity plays a crucial role.
\end{remark}

\section{Behavior of $u_\e$ in $B\setminus B_{r_{k,\e}}$}\label{sec:outer}
Next  we show the behavior on $B\setminus B_{r_{k,\e}}$. We set $u_{k+1,\e}:=u_{\e}|_{B\setminus B_{r_{k,\e}}}\in H^1_0(B)$ by zero extension. Then we have the following
\begin{proposition}\label{prop:outer} We get
\[
u_{k+1,\e}\to u_0\text{ in }H^1_0(B),
\]
as $\e\to0$ where $u_0$ is the least energy solution of \eqref{i9}. 
\end{proposition} 
First observe that we have already proved 
\[
0<\lim_{\e\to0}I_\e(u_{k+1,\e})= \inf_{u\in \mathcal{N}_0}I_0(u)<2\pi
\]
by Lemma \ref{lem:energy5}. This means that the energy of $u_{k+1,\e}$  belongs to the suitable compactness region for Palais-Smale sequences \cite{A}. Although we do not ensure $\lim_{\e\to0}I_{\e}'(u_{k+1,\e})=0$,  we can accomplish the proof by the argument based Lions' concentration compactness result \cite{L}. We refer the proof in \cite{A} (and also \cite{FMR}).
\begin{proof}[Proof of Proposition \ref{prop:outer}]
Since $u_\e$ is bounded, we can assume, by choosing a sequence if necessary, that there exists a function $u_0\in H^1_0(\Omega)$ such that 
\begin{equation}
\begin{split}
&u_{\e}\rightharpoonup u_0 \text{ weakly in } H^1_0(B), \\
&u_{\e}\to u_0\text{ in }L^p(B)\text{ for all }p\ge1,\\
&u_{\e}\to u_0\text{ a.e. on }B
\end{split}\label{eq:1}
\end{equation}
as $\e\to 0$. Then, since $u_{i,\e}\rightharpoonup 0$ weakly in $H^1_0(B)$ for all $i=1,2,\cdots,k$, we also have
\begin{equation}
\begin{split}
&u_{k+1,\e}\rightharpoonup u_0 \text{ weakly in } H^1_0(B), \\
&u_{k+1,\e}\to u_0\text{ in }L^p(B)\text{ for all }p\ge1,\\
&u_{k+1,\e}\to u_0\text{ a.e. on }B,
\end{split}\label{eq:2}
\end{equation}
as $\e\to 0$. Furthermore, since $\langle I_\e'(u_{\e}),u_{\e}\rangle=0$, we get $\int_Bf_{\e}(u_{\e})u_{\e}dx$ is bounded. Then Lemma A.1 implies $f_{\e}(u_{\e})\to f_0(u_0)$ in $L^1(B)$. We claim that $u_0$ is a nonnegative weak solution of \eqref{P} with $\e=0$. In fact, for all $\psi\in C^{\infty}_0(B)$, we get by the weak convergence of $u_\e$ and $L^1(B)$ convergence of $f_{\e}(u_\e)$,  
\[
\begin{split}
0&=\lim_{n\to \infty}\left\{\int_B\nabla u_\e\nabla \psi dx-\int_Bf_{\e}(u_\e)\psi dx\right\}\\&=\int_B\nabla u_0\nabla \psi dx-\int_Bf_{0}(u_0)\psi dx.
\end{split}
\] 
By a density argument we prove the claim. Next  we shall show that there exists a constant $q>1$ such that 
\begin{equation}
\int_B|f_{\e}(u_{k+1,\e})|^qdx\text{ is bounded}.\label{eq:3}
\end{equation}
To see this, we observe that for a constant  $\beta>1$, which will be determined later, there exists  $C>0$ such that $|f_{\e}(t)|\le Ce^{\beta t^{2}}$ for all $t\in \mathbb{R}$ and small $\e>0$. 
Then for $q>1$, which will be also chosen later, we get
\[
\int_B|f_\e(u_{k+1,\e})|^qdx\le C\int_B e^{q\beta u_{k+1,\e}^2}dx=C\int_B e^{q\beta \|u_{k+1,\e}\|^2v_\e^2}dx
\]
where we set $v_\e:=u_{k+1,\e}/\|u_{k+1,\e}\|$. Notice $\|v_\e\|=1$ and $v_\e\rightharpoonup v_0$ weakly in $H^1_0(B)$ for a function $v_0$ with $0\le \|v_0\|\le1$. We claim that $v_0\not=0$. If on the contrary $v_0=0$ we get $u_0=0$. Then Lemma A.1 shows $\int_BF_{\e}(u_{k+1,\e})dx\to0$ as $\e\to0$. It follows that 
\begin{equation}0<2\lim_{\e\to0}I_\e(u_{k+1,\e})=\lim_{\e\to0} \|u_{k+1,\e}\|^2<4\pi.\label{eq:4}
\end{equation} 
Consequently we can choose $\beta,q>1$ so that  
\[
\int_B|f_\e(u_{k+1,\e})|^qdx\le C\int_B e^{q\beta \|u_{k+1,\e}\|^2v_\e^2}dx\le C\int_B e^{4\pi v_\e^2}dx
\] 
for small $\e>0$. Notice that the Trudinger Moser inequality implies that the right hand side is bounded uniformly for small $\e>0$. Now setting $q'>1$ so that $1/q+1/q'=1$, we get by the H\"{o}lder inequality that 
\[
\begin{split}
\|u_{k+1,\e}\|^2&=\int_Bu_{k+1,\e}f_\e(u_{k+1,\e})dx\\&\le \left(\int_B|u_{k+1,\e}|^{q'}dx\right)^{\frac{1}{q'}}\left(\int_B|f_\e(u_{k+1,\e})|^qdx\right)^{\frac1q}\\&\le C\left(\int_B|u_{k+1,\e}|^{q'}dx\right)^{\frac{1}{q'}}.
\end{split}
\]
for a constant $C>0$ if $\e>0$ is small enough. Hence, we get $u_{k+1,\e}\to 0$ in $H^1_0(B)$ by \eqref{eq:2}. This contradicts  \eqref{eq:4}. Therefore, we can assume $0<\|v_0\|<1$. (If $\|v_0\|=1$, we finish the proof.) Then Lions' concentration compactness lemma (Theorem I.6 in \cite{L}) proves 
\begin{equation}
\int_B e^{4\pi p v_\e^2}dx \text{ is bounded for all }p<\frac{1}{1-\|v_0\|^2}.\label{eq:L}
\end{equation}
Now recalling the facts that $\lim_{\e\to0}I_\e(u_{k+1,\e})<2\pi$, $f_{0}(t)t-2F_{0}(t)\ge 0$ for all $t\in \mathbb{R}$ 
 and $\langle I_{0}'(u_0),u_0\rangle=0$, we get a constant $\delta\in (0,1)$ such that 
\[
\begin{split}
4\pi(1-\delta)&=2\lim_{\e\to0}I_\e(u_{k+1,\e})\\&=\lim_{\e\to 0}\|u_{k+1,\e}\|^2-2\int_BF_{0}(u_0)dx-\langle I_{0}'(u_0),u_0\rangle\\&\ge  \lim_{\e\to 0}\|u_{k+1,\e}\|^2-\|u_0\|^2\\&=\lim_{\e\to 0}\|u_{k+1,\e}\|^2(1-\|v_0\|^2).
\end{split}
\]
This shows
\[
q\beta \lim_{\e\to 0}\|u_{k+1,\e}\|^2\le\frac{4\pi q\beta(1-\delta)}{1-\|v_0\|^2}.\]
Put $p:=q\beta(1-\delta)/(1-\|v_0\|^2)$. Then, we can choose $\beta,q>1$ so that $p<1/(1-\|v_0\|^2)$ and
\[
\int_B|f_\e(u_{k+1,\e})|^qdx\le C\int_B e^{4\pi p v_\e^2}dx,
\]
for small $\e>0$. Then \eqref{eq:L} proves \eqref{eq:3}. Now choose $u_{0,\e}\in H_0^1(B)$ such that $u_{0,\e}=0$ on $\overline{B_{r_{k,\e}}}$ and $u_{0,\e}\to u_0$ in $H^1_0(B)$ as $\e\to0$. (Define, for example, $u_{0,\e}:=\phi_{r_{k,\e},1}u_0$ where $\phi_{r_{k,\e},1}$ is a cut off function defined as in Section 2.) Then integration by parts gives that
\[
\begin{split}
\int_B \nabla u_{k+1,\e}\nabla(u_{k+1,\e}-u_{0,\e})dx&=\int_{B\setminus B_{r_{k,\e}}} (-\Delta u_{k+1,\e})(u_{k+1,\e}-u_{0,\e})dx\\&=\int_Bf_{\e}(u_{k+1,\e})(u_{k+1,\e}-u_{0,\e})dx.
\end{split}
\]
Now again let $q'>1$ be a constant such that $q^{-1}+q'^{-1}=1$. Then setting $o(1)\to0$ as $\e\to0$, and  using the H\"{o}lder inequality, \eqref{eq:3}, and \eqref{eq:2},  we get 
\[
\begin{split}
\|u_{k+1,\e}\|^2-\|u_0\|^2&=\int_{B}\nabla u_{k+1,\e}\nabla (u_{k+1,\e}-u_{0,\e})dx+o(1)\\&=\int_Bf_{\e}(u_{k+1,\e})(u_{k+1,\e}-u_{0,\e})dx+o(1)\\
                                                &\le \left(\int_B|f_{\e}(u_{k+1,\e})|^qdx\right)^{\frac{1}{q}}\left(\int_B|u_{k+1,\e}-u_{0,\e}|^{q'}dx\right)^{\frac{1}{q'}}+o(1)\\&\to0
\end{split}
\]
as $\e\to \infty$. Hence we get $u_{k+1,\e}\to u_0$ in $H^1_0(B)$ as $\e\to 0$. Finally, Lemma \ref{lem:energy5} proves that $u_0$ is the least energy solution of \eqref{i9}. This completes the proof.
\end{proof}
\begin{remark}\label{rem:outer} From the result above, we get $\|u_{k+1,\e}\|_{L^{\infty}((r_{k,\e},1))}$ is bounded. To see this, observe that the strong convergence of $u_{k+1,\e}$ implies that for all $q>1$,  $e^{u_{k+1,\e}^2}$ is bounded in $L^q(B)$ uniformly for small $\e>0$. Set $r_{k+1,\e}^*\in (r_{k,\e},1)$ so that $u_{r_{k+1},\e}(r_{k+1,\e}^*)=\|u_{k+1,\e}\|_{L^{\infty}((r_{k,\e},1))}$. Then we get
\[
\begin{split}
|u_{k+1,\e}(r_{k+1,\e}^*)|&=\left|\int_{r_{k+1,\e}^*}^1f_\e(u_{k+1,\e})r\log{r}dr\right|\\&\le \left(\int_{r_{k+1,\e}^*}^1f_\e(u_{k+1,\e})^2rdr\right)^{\frac12}\left(\int_{r_{k+1,\e}^*}^1r\log^2{r}dr\right)^{\frac12}\\
&\le C
\end{split}
\]
for a constant $C>0$ if $\e>0$ is sufficiently small. This proves the claim.
\end{remark}
\begin{remark}\label{rem:outer2} The previous remark shows
\[
\lim_{\e\to0}r_{k,\e}u_{\e}'(r_{k,\e})=0.
\]
To show this, set $r_{k+1,\e}^*\in (r_{k,\e},1)$ as above. First observe that $r_{k+1,\e}^*\to 0$ as $\e\to0$. If not, we have a constant $r_0\in(0,1)$ such that $r_{k+1,\e}^*\to r_0$ as $\e\to0$ by choosing a sequence if necessary. Then since $u_{k+1,\e}\to u_0$ a.e. on $B$, we have $u_0(r)\le u_0(r_0)$ for a.a. $r\in(0,r_0)$. But, since $u_0$ is a positive radial solution of \eqref{i9}, the result in \cite{GNN} shows  $u'(r)<0$ for all $r\in (0,1)$. This is a contradiction. Finally,  integrating \eqref{i11} over $(r_{k,\e},r_{k+1,\e}^*)$, we get by the previous remark that
\[
|r_{k,\e}u_{\e}'(r_{k,\e})|=\left|\int_{r_{k,\e}}^{r_{k+1,\e}^*}f_{\e}(u_{\e})rdr\right|\le C r_{k+1,\e}^*.
\]
for some constant $C>0$. This completes the proof.
\end{remark}

\section{Proof of main theorems}\label{sec:proof}
We finally conclude the proof of our main theorems. 
\begin{proof}[Proof of Theorem \ref{i10}] The proof of \eqref{i98} is given in Lemma \ref{lem:kthzero},  \eqref{i97} is shown in
Corollary \ref{cor:ball}, Proposition \ref{prop:annulus} and Proposition  \ref{prop:outer}, \eqref{i96} is shown in
Corollary \ref{cor:ball}, Proposition \ref{prop:annulus} and Proposition  \ref{prop:outer}. So we have only to show  \eqref{i99}, i.e.
\[
u_\e\to (-1)^ku_0\text{ in }C^2_{\text{loc}}((0,1]),
\]
as $\e\to0$. To prove this, we may assume that  $u_{\e}$ satisfies 
\begin{equation}
\begin{cases}
-u_{\e}''-\frac{1}{r}u_{\e}'=\lambda u_{\e}e^{u_{\e}^2+u_{\e}^{1+\e}} ,\ u_{\e}>0\ \text{ in }(r_{k,\e},1),\\
u_{\e}(r_{k,\e})=u_{\e}(1)=0.\label{2}
\end{cases}
\end{equation}
Now choose any compact subset $K\subset\subset (0,1]$. For all $r\in K$, we may suppose $r_{k,\e}<r$ by Lemma \ref{lem:kthzero}. Then multiplying \eqref{2} by $r$ and integrating over $(r_{k,\e},r)$ we have
\[
ru_\e'(r)=r_{k,\e}u_\e'(r_{k,\e})-\int_{r_{k,\e}}^rf_\e(u_\e)rdr.
\]
Then Remark \ref{rem:outer2} and Lemma \ref{lem:F} prove that $ru_{\e}'$ is bounded uniformly on $K$ for small $\e$. In particular, $u_\e'$ is bounded uniformly on $K$  for small $\e$. Furthermore, since  for any $r\in K$, we have
\[
u_\e(r)=-\int_{r}^1u_\e'(s)ds,
\]
we derive $\|u_\e\|_{C^1(K)}$ is bounded uniformly for small $\e$. Then the Arzela-Ascoli theorem ensures $u_\e\to u_0$ uniformly on $K$ as $\e\to0$. Finally, using \eqref{2}, we show $u_\e\to u_0$ in $C^2(K)$ as $\e\to0$. This finishes the proof.
\end{proof}

\begin{proof}[Proof of Theorem \ref{i12}] For $i=0$ the proof is given in Proposition \ref{prop:ball}. The case $i=1,..,k$ is considered in Proposition \ref{prop:annulus}.
\end{proof}

Using the blow up results above, we get the following remark.
\begin{remark}\label{rem} We have
\begin{equation}
\lim_{\e\to0}\frac{\|u_{\e}\|_{L^{\infty}((r_{i,\e},r_{i+1,\e}))}}{\|u_{\e}\|_{L^{\infty}((r_{i-1,\e},r_{i,\e}))}}=0\label{eq:remark}
\end{equation}
for all $i=1,2,\cdots,k$. Let us show the proof. For $i=k$, the proof is obvious by Corollary \ref{cor:ball}, Lemma \ref{lem:infty} and Remark \ref{rem:outer}. Then for $i=1,2,\cdots,k-1$, set $M_{i,\e}\in[0,1)$ so that $|u_\e(M_{i,\e}r_{i,\e})|=\|u_{\e}\|_{L^{\infty}((r_{i-1,\e},r_{i,\e}))}$, $r_{i,\e}^*:=M_{i,\e}r_{i,\e}$ and $\gamma_{i,\e}=2u_\e(r_{i,\e}^*)f_{\e}(u_\e(r_{i,\e}^*))>0$. Then integrating \eqref{i11} over $(r_{i,\e}^*,r_{i+1,\e}^*)$ shows
\[
\int_{r_{i,\e}^*}^{r_{i,\e}}f_{\e}(u_\e)rdr=-\int_{r_{i,\e}}^{r_{i+1,\e}^*}f_{\e}(u_\e)rdr.
\]
Hence putting $v_{i,\e}(r)=|u_{i,\e}(r_{i,\e}r)|$, $v_{i+1,\e}(r)=|u_{i+1,\e}(r_{i+1,\e}r)|$ and $z_{j,\e}(r)=2|u_{\e}(r_{j,\e}^*)|(|v_{j,\e}(\gamma_{j,\e}r+M_{j,\e})|-|u_{\e}(r_{j,\e}^*))|)$ for $j=i,i+1$, we get by the blow up procedure as above,
\[
\begin{split}
&\frac{1}{2|u_{\e}(r_{i,\e}^*)|}\times
\\&\int_{0}^{\frac{1-M_{i,\e}}{\gamma_{i,\e}}}\left(\frac{z_{i,\e}(r)}{2v_{i,\e}^2(M_{i,\e})}+1\right)\times\\
&\ \ \ \ \ \ \ \ \ \ \ \ \ \ \ \ e^{z_{i,\e}(r)\left(\frac{z_{i,\e}(r)}{4v_{i,\e}^2(M_{i,\e})}+1\right)+v_{i,\e}^{1+\e}(M_{i,\e})\left(\left|\frac{z_{i,\e}(r)}{2v_{i,\e}^2(M_{i,\e})}+1\right|^{1+\e}-1\right)}\left(\frac{M_{i,\e}}{\gamma_{i,\e}}+r\right)dr\\
&=\frac{1}{2|u_\e(r_{i+1,\e}^*)|}\times\\
&\int_{\frac{\frac{r_{i,\e}}{r_{i+1,\e}}-M_{i+1,\e}}{\gamma_{i+1,\e}}}^{0}\left(\frac{z_{i+1,\e}(r)}{2v_{i+1,\e}^2(M_{i+1,\e})}+1\right)\times\\
& \ \  e^{z_{i+1,\e}(r)\left(\frac{z_{i+1,\e}(r)}{4v_{i+1,\e}^2(M_\e)}+1\right)+v_{i+1,\e}^{1+\e}(M_{i+1,\e})\left(\left|\frac{z_{i+1,\e}(r)}{2v_{i+1,\e}^2(M_{i+1,\e})}+1\right|^{1+\e}-1\right)}\left(\frac{M_{i+1,\e}}{\gamma_{i+1,\e}}+r\right)dr\\
&=o\left(\frac{1}{2|u_\e(r_{i+1,\e}^*)|}\right)
\end{split}
\]
as $\e\to0$ since $\lim_{\e\to0}\frac{r_{i,\e}/r_{i+1,\e}-M_{i+1,\e}}{\gamma_{i+1,\e}}=0$. 
Finally using our blow up results and the Fatou lemma for the integral on the left hand side, we get  the desired conclusion.
\end{remark}
\appendix
\section{Some basic facts}\label{sec:basic}
In the following, let $(\e_n)\subset \R^+$ be any sequence such that $\e_n\to0$ as $n\to \infty$.
\begin{lemma}\label{lem:F} Let $(u_n)\subset H^1_0(B)$ be a bounded sequence such that  $u_n\rightharpoonup u$ weakly in $H^1_0(B)$ and $u_n\to u$ a.e. on $B$ as $n\to \infty$ for a function $u$. Furthermore, assume 
\[
\sup_{n}\int_Bf_{\e_n}(u_n)u_n dx< \infty.\]
Then we have
\[
\lim_{n\to \infty}\int_B f_{\e_n}(|u_n|)dx= \int_B f_{0}(|u|)dx
\]
and
\[
\lim_{n\to \infty}\int_B F_{\e_n}(u_n)dx= \int_B F_{0}(u)dx.
\]
\end{lemma}
\begin{proof}
Similar to the proof of 4) of Lemma 3.1 in \cite{A1}.
\end{proof}
\begin{lemma}\label{lem:energylb0}
We have
\[
\liminf_{\e\to0}\inf_{u\in \mathcal{N}_\e}I_\e(u)>0.
\]
\end{lemma}
\begin{proof} If not, we have sequences $(\e_n)\subset \R^+$ and $(u_n)\subset \mathcal{N}_{\e_n}$ such that $\lim_{n\to \infty}I_{\e_n}(u_n)=0$. Then 
since $\lambda<\lambda_1$, analogously with Step 1 in the proof of Lemma 3.4 in \cite{A1}, we can get a contradiction. This proves the lemma.
\end{proof} 
\begin{lemma}\label{lem:cpt} Let $(\mu_n)\subset \R^+$ and $(u_n)\subset H^1_0(B)$ be sequences such that $\mu_n\le 1$ for all $n$ and further, 
\[
\begin{split}
&J_n(u_n):=\int_B|\nabla u_n|^2dx-\mu_n\int_BF_{\e_n}(u_n)dx\to c\in (0,2\pi)\text{ and }\\
&J_{n}'(u_n)\to0 \text{ in }H^{-1}(B),
\end{split}
\]
as $n\to \infty$. Then $u_n\to u$ in $H^1_0(B)$ up to a subsequence.
\end{lemma}
\begin{proof}
Similar to 1) on p404 in \cite{A1}.
\end{proof}

\end{document}